\documentclass[a4paper,11pt,3p]{article}
\usepackage{amsmath,amsthm,amscd,amsfonts,amssymb,enumitem}
\usepackage{latexsym}

\usepackage[colorlinks=true, linkcolor=black, anchorcolor=black, citecolor=black, filecolor=black, urlcolor=black]{hyperref}

\newtheorem{thm}{Theorem}[section]
\newtheorem{lemma}[thm]{Lemma}

\newtheorem{rem}[thm]{Remark}

\numberwithin{equation}{section}
\newcommand\blfootnote[1]{%
	\begingroup
	\renewcommand\thefootnote{}\footnote{#1}%
	\addtocounter{footnote}{-1}%
	\endgroup
}

\def\NN{\mathbb N}
\def\ZZ{\mathbb Z}
\def\RR{\mathbb R}

\def\EE{\mathbb E}

\def\aa{\mathcal A}

\def\ee{\mathcal E}

\def\oo{\mathcal O}

\def\intl{\int\limits}
\def\supl{\sup\limits}
\def\suml{\sum\limits}

\title{Random Sampling in reproducing kernel subspaces of $L^p(\RR^n)$}
\author{Dhiraj Patel\thanks{Email id: dpatel.iitd@gmail.com}, Sivananthan Sampath\thanks{Email id: siva@maths.iitd.ac.in}\\ 
		Department of Mathematics, Indian Institute of Technology Delhi,\\ New Delhi-110016, India}
\date{}

\begin{document}
	\maketitle
	
	\begin{abstract}
		In this paper, we study random sampling in reproducing kernel space $V$, which is a range of an idempotent integral operator. Under certain decay condition on the integral kernel, we show that any element in $V$ can be approximated by an element in a finite-dimensional subspace of $V$. Moreover, we prove with overwhelming probability that random points uniformly distributed over a cube $C$ is stable sample for the set of functions concentrated on $C$.
	\end{abstract}
	
	
	\section{Introduction}
	Sampling problem is a fundamental interest in signal processing and digital communication. It deals to find a discrete sample set which allows to convert an analog signal into a digital signal and vice-versa without losing any information. Of course, the problem is well-posed only when we impose some conditions on signals(functions). For example, the Shannon-sampling theorem states that if $f\in PW_{[-\frac{1}{2},\frac{1}{2}]}(\RR),$ the space of functions in $L^2(\RR)$ whose Fourier transform supported in $[-\frac{1}{2},\frac{1}{2}]$, then  $f$ can be reconstructed by its uniform sample values $\{f(k):k\in \ZZ\}$ and reconstruction is given by $$f(x)=\sum_{k\in \ZZ} f(k)\frac{\sin{\pi(x-k)}}{\pi(x-k)}.$$ 
	Further, it is well-known by Kedac's 1/4-theorem that if $X=\{x_k : |x_k-k|\leq L<\frac{1}{4} \}$, then every $f\in PW_{[-\frac{1}{2},\frac{1}{2}]}(\RR)$ can be reconstructed from its sample values on $X$. Mathematically, the sampling problem can be stated as follows.\par
	
	Given a closed subspace $V$ of $L^p(\RR^n)$, find a countable set $\Gamma  \subset \RR^n$ such that 
	\begin{equation}
	A\|f\|_{L^p(\RR^n)}^p\leq \suml_{\gamma\in \Gamma} |f({\gamma})|^p \leq B\|f\|_{L^p(\RR^n)}^p ~~~\text{for all } f\in V, 
	\label{eqn:SE}
	\end{equation}
	for some $A,B>0$.  This is equivalent to say that the sampling operator $S: f \mapsto (f(\gamma))_{\gamma\in \Gamma}$ from the space $V$ into $\ell^p(\Gamma)$ is continuous and the corresponding inverse operator $S^{-1}:Range(S)\rightarrow V$ is also continuous. Hence any $f\in V$ can be reconstructed from its sample values on $\Gamma$, and the set $\Gamma$ is called a stable set of sample, or simply stable sampling for the space $V$. For the detailed study of sampling problem refer to \cite{butzer1992sampling,aldroubi2001nonuniform,olevskii2016functions}. \par
		
	For the Paley-Wiener space $PW_{[a,b]}(\RR)$, the set of stable sampling is characterized by Beurling density condition. However, a similar characterization is not true in higher dimension. In particular,  the sufficient Beurling density condition for stable sampling in $\RR^n \,(n\geq 2)$ does not hold, see \cite[Section 5.7]{olevskii2016functions}. At the same time, the Paley-Wiener theorem is still valid for the convex spectra in $\RR^n$, but for $n\geq 2$ the zero set of a function in $PW_{\Omega}(\RR^n)$ is analytic manifold, where $\Omega$ is a convex subset of $\RR^n$. Hence the classical result on the density of zeros of an entire function of exponential type is no longer valid. So the non-uniform sampling in higher dimension is still difficult to solve. These difficulties motivate to study the sampling problem in probabilistic framework.\par 
	
	Random sampling method has been used frequently in the field of image processing \cite{chan2014monte}, learning theory \cite{poggio2003mathematics,cucker2002mathematical} and compressed sensing \cite{eldar2009compressed}. Bass and Gr{\"o}chenig  \cite{bass2005random} studied random sampling for multivariate trigonometric polynomial. Later, Cand{\'e}s, Romberg and Tao  \cite{candes2006robust,candes2006stable} investigated reconstruction of sparse trigonometric polynomial from a few random samples. Smale and Zhou \cite{smale2004shannon} studied the function reconstruction error from its random samples satisfying \eqref{eqn:SE} in a reproducing kernel Hilbert space.
	
	Note that for ``nice" functions $f$, the sample value $f(\gamma)$ is close to 0 when  $\gamma$ is large value. Therefore, the sample value may not significantly contribute  to sampling inequality for large sample points. Moreover, it is shown by Bass and Gr\"{o}chenig \cite{bass2010random} that for each random samples identically and uniformly distributed over each cube $k+[0,1]^n$ in $\RR^n$, the sampling inequality \eqref{eqn:SE} fails almost surely for Paley-Wiener space. For these reasons, they considered random sample points from the compact set $C_R=\left[ -\frac{R}{2},\frac{R}{2} \right]^n$ and proved that the sampling inequality \eqref{eqn:mainsaminq} hold for the functions concentrated on $C_R$ with high probability, see \cite{bass2010random,bass2013relevant}. Then the result was generalized  by F\"{u}hr and Xian \cite{fuhr2019relevant} for finitely generated shift-invariant subspace $V$ of $L^2(\RR^n)$, and a further generalization for $L^p$-norm was studied by Yang and Wei \cite{yang2013random,yang2019random}. Later, Yang and Tao \cite{tao2019random} investigated random sampling for the space of continuous functions with bounded derivative. \par
	
	In this paper, we study random sampling on a closed subspace $V$ of $L^p(\RR^n)$ which is defined as the image of an idempotent integral operator. More precisely, we derive the random sampling inequality for the set $V^{\star}(R,\delta)=\big\{f\in V : (1-\delta)\|f\|_{L^p(\RR^n)}^p\leq \int_{[-\frac{R}{2},\frac{R}{2}]^n} |f(x)|^p dx \big\}$, and prove the following main theorem.
	
	\begin{thm}
		\label{thm:mainresult}
		Assume that $\{x_j : j\in \NN\}$ is a sequence of i.i.d. random variables that are uniformly distributed over the cube $C_R=\left[ -\frac{R}{2},\frac{R}{2} \right]^n$ and $0<\mu<1-\delta$. Then there exist $a,b>0$ such that the sampling inequality
		\begin{equation}
		\frac{r}{R^n}(1-\mu-\delta)\|f\|_{L^p(\RR^n)}^p\leq \sum_{j=1}^{r} |f(x_j)|^p \leq \frac{r}{R^n}(1+\mu)\|f\|_{L^p(\RR^n)}^p
		\label{eqn:mainsaminq}
		\end{equation}
		holds for every $f\in V^{\star}(R,\delta)$ with the probability at least $1-2a\exp\left( -\frac{b}{pk^{p-1}R^n} \frac{r\mu^2}{12+\mu} \right)$, where $k=\supl_{x\in C_R} \|K(x,\cdot)\|_{L^{p'}(\RR^n)}$.
	\end{thm}
	
	 The paper is organized as follow. In Section \ref{Covering Number}, we introduce our hypothesis space $V$ and prove that any element in $V$ can be approximated by a finite-dimensional subspace of $V$ with respect to $\|\cdot\|_{L^p(C_R)}$. Further, we show that the set of functions in $V^{\star}(R,\delta)$ with unit norm is totally bounded and estimate the number of open balls which covers the set. In Section \ref{Random Sampling}, we define independent random variables with respect to given random samples, and then using Bernstein's inequality we prove the main result.
	
	\section{Assumption and Covering Number} \label{Covering Number}
%
	
	Let $T$ be an idempotent  integral operator from $L^p(\RR^n)$ to $L^p(\RR^n)$ defined by
	\begin{equation}
	Tf(x):=\int_{\RR^n} K(x,y)f(y)dy ~ \mbox{ satisfy } T^2=T,
	\label{eqn:Operator}
	\end{equation} where $1\leq p<\infty$, and the integral kernel $K$ is symmetric and satisfy regularity condition
	\begin{equation}
	\lim\limits_{\varepsilon\rightarrow 0}\left\| \sup\limits_{z\in \RR^n}|osc_{\varepsilon}(K)(\cdot+z,z)|\right\|_{L^1(\RR^n)}=0,
	\label{eqn:oscK}
	\end{equation}
	with decay
	\begin{equation}
	|K(x,y)|\leq \frac{C}{(1+\|x-y\|_{1})^{\alpha}}, ~~~~~\alpha >\frac{n}{p'}+n+1 ~\text{and } C>0,
	\label{eqn:DKkrnl}
	\end{equation}
	where $\|x\|_1:=\sum\limits_{i=1}^{n}|x(i)|,~ x:=(x(1),x(2),\cdots,x(n))\in \RR^n$, $\frac{1}{p}+\frac{1}{p'}=1$, and $$osc_{\varepsilon}(K)(x,y)=\supl_{x',y'\in [-\varepsilon,\varepsilon]^n}|K(x+x',y+y')-K(x,y)|.$$
	
	Under these assumptions,  we see that $\supl_{x\in \RR^n} \|K(x,\cdot)\|_{L^1(\RR^n)}$ exists and the associated integral operator $T$ is bounded. Moreover, the kernel $K$ satisfy the off-diagonal decay condition
	\begin{equation}
	\big\|\sup\limits_{z\in \RR^n} |K(\cdot+z,z)|\big\|_{L^1(\RR^n)}< \infty.
	\label{eqn:offdiagonaldecay}
	\end{equation}
	
	A reproducing kernel Banach space is a Banach space $M$ of functions on a set $\Omega$  such that the point evaluation functional $f \mapsto f(x)$ is continuous for each $x\in \Omega$ i.e., for every $x\in \Omega$, there exists $C_x>0,$ such that $|f(x)|\leq C_x \|f\|,$ for all $f\in M.$ Let us consider the space $V:=Range(T)$. It is easily verifiable that the space $V$ is closed and reproducing kernel subspace of $L^p(\RR^n)$.

Originally, Nashed and Sun \cite{sun2010sampling} proposed the space $V$ as a general model to study the sampling inequality \eqref{eqn:SE}. They showed that  if the integral kernel $K$ satisfies off-diagonal decay condition \eqref{eqn:offdiagonaldecay} and regularity condition \eqref{eqn:oscK}, then there exists a discrete stable sampling for the space $V$. For more details about the space $V$,  we refer the reader to \cite{sun2010sampling}.
 
	Note that the sampling inequality \eqref{eqn:mainsaminq} is true for $f\in V^{\star}(R,\delta)$ if and only if it is true for $f\in V^{\star}(R,\delta)$ with $\|f\|_{L^p(\RR^n)}=1$. Hence it is enough to prove the \hyperref[thm:mainresult]{Theorem \ref{thm:mainresult}} for the set $$V(R,\delta)=\big\{ f\in V ~:~ (1-\delta)\leq \int_{C_R} |f(x)|^p dx ~~\text{and } \|f\|_{L^p(\RR^n)}=1 \big\}.$$
	
	The ``key step" of \hyperref[thm:mainresult]{Theorem \ref{thm:mainresult}} is to prove $V(R,\delta)$ is totally bounded with respect to $\|\cdot\|_{L^\infty(C_R)}$. In the previous works, Bass and Gr\"{o}chenig \cite{bass2010random} used spectral decomposition of truncated Fourier transform on band-limited functions and eigenvalue decay condition of prolate spheroidal wave functions. Yang and Wei \cite{yang2013random} considered shift-invariant space generated by compactly supported function which implies  $V(R,\delta)$ is a subset of finite-dimensional space. In \cite{fuhr2019relevant}, F\"uhr and Xian calculated the maximum number of eigenvalues of some self-adjoint operator which are greater than $\frac{1}{2}$ and used a similar method as in \cite{bass2013relevant}. For finitely generated shift-invariant space \cite{yang2019random}, Yang assumed a fixed decay on each generator and approximated any function in $V(R,\delta)$ by a function in some finite-dimensional subspace of $V.$ In this paper, the considered space generalizes the existing model spaces. Moreover, the function space need not have finite generators. The key idea is to use the existence of stable sampling set for the space $V$, which was proved in \cite{sun2010sampling}. This allows us to represent any functions in $V$ via some frame sequence, and then we estimate the decay of frame sequence using decay property of the integral kernel. Using these estimates, we are able to show that $V(R,\delta)$ is totally bounded with respect to $\|\cdot\|_{L^\infty(C_R)}$.   
	
	A collection of points $U=\{ u : u\in \RR^n \}$ is \textit{relatively separated} if $$\beta=\inf_{\underset{u\neq u'}{u,u'\in U}} \| u-u' \|_1>0,$$ and $\beta$ is called \textit{gap} of the set $U$.
	
	As the kernel $K$ of the integral operator $T$ defined in \eqref{eqn:Operator} satisfies off-diagonal decay condition \eqref{eqn:offdiagonaldecay} and  regularity condition \eqref{eqn:oscK}, then from \cite[Theorem A.2.]{sun2010sampling} there exist a relatively separated set $\Gamma=\{\gamma : \gamma \in \RR^n\}$ with positive gap $\eta \, (<\frac{2}{n})$, and two families $\Phi:=\{\phi_\gamma\}_{\gamma\in \Gamma}\subseteq L^p(\RR^n)$ and $\tilde{\Phi}:=\{\tilde{\phi}_\gamma\}_{\gamma\in \Gamma}\subseteq L^{p'}(\RR^n)$ such that for any $f\in V$ can be written as 
	\begin{equation}
	f(x)=\sum_{\gamma\in \Gamma} \langle f,\tilde{\phi}_{\gamma} \rangle \phi_{\gamma}(x),
	\label{eqn:fseries}
	\end{equation}
	where for each $\gamma \in \Gamma$, $\phi_{\gamma}$ is given by
	\begin{equation}
	\phi_{\gamma}(x)={\eta}^{-\frac{n}{p}}\int_{C_{\eta}} K(\gamma+z,x)dz, ~~x\in \RR^n,
	\label{eqn:phigamma}
	\end{equation}
	and $\{\tilde{\phi}_{\gamma} : \gamma\in \Gamma\}$ forms p-frame for $V$, i.e. there exist $A,B>0$ such that
	\begin{equation}
	A\|f\|_{L^p(\RR^n)}^p\leq \sum_{\gamma\in \Gamma} |\langle f,\tilde{\phi}_{\gamma} \rangle|^p\leq B\|f\|_{L^p(\RR^n)}^p, \hspace{0.5cm} \text{for all } f\in V.
	\label{eqn:pframe}
	\end{equation}
	
	By \eqref{eqn:fseries} and \eqref{eqn:pframe}, any $f$ in $V$  is of the form $\suml_{\gamma\in \Gamma} c_{\gamma}\phi_{\gamma}$ for some $(c_{\gamma})\in \ell^p$. Now, our interest is to approximate any function in $V$ by an element in a finite-dimensional space.  In the following lemma, given $f$ in $V$, we determine the sufficient condition on real number $N$ such that the truncated series $\suml_{\gamma\in \Gamma\cap C_N} c_{\gamma}\phi_{\gamma}$ is close to $f$.\par

	Before we move to the lemma, we define the subspace $V_N$ of $V$ by $$V_N=\Big\{ \suml_{\gamma\in \Gamma\cap [-\frac{N}{2},\frac{N}{2}]^n} c_{\gamma}\phi_{\gamma} : c_{\gamma}\in \RR \Big\},$$ where $N>R+\frac{2}{n}$ be a positive real.
	
	\begin{lemma}
		\label{lemma:p-approx}
		For a given $\epsilon>0$ and $f\in V$, choose $N>R+\frac{2}{n}+\frac{2}{n}\left[ \frac{4^nB^{(p'-1)}\|f\|_{L^p(\RR^n)}^{p'}C^{p'}R^{n(p'-1)}}{w_\alpha\epsilon^{p'}} \right]^{\frac{1}{\alpha p'-n}}$, then $\|f-\suml_{\gamma\in \Gamma\cap [-\frac{N}{2},\frac{N}{2}]^n} \langle f,\tilde{\phi}_{\gamma} \rangle \phi_{\gamma}\|_{L^p(C_R)}<\epsilon$.
	\end{lemma}
	\begin{proof}
		Given $f\in V$ we consider $f_N\in V_N$ by
		\begin{equation}
		f_N(x)=\suml_{\gamma\in \Gamma\cap [-\frac{N}{2},\frac{N}{2}]^n} \langle f,\tilde{\phi}_{\gamma} \rangle \phi_{\gamma}(x).
		\label{eqn:fNseries}
		\end{equation}
		Then by \eqref{eqn:fseries}, \eqref{eqn:fNseries} and \eqref{eqn:pframe} we have
		\begin{align*}
			\|f-f_N\|_{L^p(C_R)}^p &=\intl_{C_R} |f(x)-f_N(x)|^p dx\\
			&= \intl_{C_R}\big|\sum_{\gamma\in \Gamma\smallsetminus [-\frac{N}{2},\frac{N}{2}]^n} \langle f,\tilde{\phi}_{\gamma} \rangle \phi_{\gamma}(x)\big|^p dx\\
			&\leq \intl_{C_R}  \sum_{\gamma\in \Gamma\smallsetminus [-\frac{N}{2},\frac{N}{2}]^n} |\langle f,\tilde{\phi}_{\gamma} \rangle|^p \Big(\sum_{\gamma\in \Gamma\smallsetminus [-\frac{N}{2},\frac{N}{2}]^n} |\phi_{\gamma}(x)|^{p'}\Big)^{\frac{p}{p'}} dx\\
			&\leq B\|f\|_{L^p(\RR^n)}^p\int\limits_{C_R} \Big(\sum_{\gamma\in \Gamma\smallsetminus [-\frac{N}{2},\frac{N}{2}]^n} |\phi_{\gamma}(x)|^{p'}\Big)^{\frac{p}{p'}} dx.
		\end{align*}
		In order to estimate the upper bound of the series, we derive the bound for $\phi_{\gamma}$ using the conditions \eqref{eqn:phigamma} and \eqref{eqn:DKkrnl}.
		\begin{align*}
		|\phi_{\gamma}(x)|&\leq \eta^{-\frac{n}{p}}\intl_{C_{\eta}} \frac{C}{(1+\|\gamma+z-x\|_1)^{\alpha}}dz\\
		&\leq \eta^{-\frac{n}{p}}\intl_{C_{\eta}} \frac{C}{(1+\|\gamma-x\|_1-\|z\|_1)^{\alpha}}dz\\
		&\leq \eta^{-\frac{n}{p}}\intl_{C_{\eta}} \frac{C}{(1-\frac{n\eta}{2}+\|\gamma-x\|_1)^{\alpha}}dz\\
		&\leq \eta^{\frac{n}{p'}} \frac{C}{(1-\frac{n\eta}{2}+\|\gamma-x\|_1)^{\alpha}}.
		\end{align*}
		
		Hence,
		\begin{align*}
			\suml_{\gamma\in \Gamma\smallsetminus [-\frac{N}{2},\frac{N}{2}]^n} |\phi_{\gamma}(x)|^{p'}&=\left(\frac{2}{\eta}\right)^n\suml_{\gamma\in \Gamma\smallsetminus [-\frac{N}{2},\frac{N}{2}]^n} |\phi_{\gamma}(x)|^{p'}\left(\frac{\eta}{2}\right)^n\\
			&\leq \left(\frac{2}{\eta}\right)^n\suml_{\gamma\in \Gamma\smallsetminus [-\frac{N}{2},\frac{N}{2}]^n} \frac{{\eta}^nC^{p'}}{(1-\frac{n\eta}{2}+\|x-\gamma\|_{1})^{\alpha p'}}\left(\gamma -\gamma +\frac{\eta}{2}\right)^n\\
			&\leq 2^nC^{p'}\suml_{\gamma\in \Gamma\smallsetminus [-\frac{N}{2},\frac{N}{2}]^n} \int_{B_{\frac{\eta}{2}}}\frac{dy}{(1-\frac{n\eta}{2}+\|x-y\|_{1})^{\alpha p'}},\\
		\end{align*}
		where $B_{\frac{\eta}{2}}$ is cube of length $\frac{\eta}{2}$ containing $\gamma$. Since $\Gamma$ is relatively separated set with gap $\eta$, we get
		
		\begin{align*}
			\suml_{\gamma\in \Gamma\smallsetminus [-\frac{N}{2},\frac{N}{2}]^n} |\phi_{\gamma}(x)|^{p'}&\leq 2^nC^{p'} \intl_{\RR^n\smallsetminus [-\frac{N-\eta}{2},\frac{N-\eta}{2}]^n} \frac{dy}{(1-\frac{n\eta}{2}+\|x-y\|_{1})^{\alpha p'}}\\
			&\leq 2^nC^{p'} \intl_{\RR^n\smallsetminus [-\frac{N-\eta-R}{2},\frac{N-\eta-R}{2}]^n} \frac{dy}{(1-\frac{n\eta}{2}+\|y\|_{1})^{\alpha p'}}\\
			&\leq 2^nC^{p'}\times 2^n\intl_{[\frac{N-\eta-R}{2},\infty)^n} \frac{dy}{\|y\|_{1}^{\alpha p'}}\\
			&= 4^nC^{p'}\frac{1}{w_\alpha\left( \frac{N-\eta-R}{2}n \right)^{\alpha p'-n}},
		\end{align*}
		where $w_\alpha=(\alpha p'-1)(\alpha p'-2)\cdots(\alpha p'-n).$\\
		Hence,
		\begin{align*}
			\|f-f_N\|_{L^p(C_R)}^p &\leq B\|f\|_{L^p(\RR^n)}^p\int\limits_{C_R} \Big(4^nC^{p'}\frac{1}{w_\alpha\left( \frac{N-\eta-R}{2}n \right)^{\alpha p'-n}}\Big)^{\frac{p}{p'}} dx\\
			&= 4^{\frac{np}{p'}} B\|f\|_{L^p(\RR^n)}^pC^{p}\frac{R^n}{w_\alpha^{\frac{p}{p'}}\left( \frac{N-\eta-R}{2}n \right)^{(\alpha-\frac{n}{p'})p}}\\
			&\leq 4^{\frac{np}{p'}} B\|f\|_{L^p(\RR^n)}^pC^{p}\frac{R^n}{w_\alpha^{\frac{p}{p'}}\left( \frac{N-R}{2}n-1 \right)^{(\alpha-\frac{n}{p'})p}}, ~~\text{as }\eta<\frac{2}{n}.
		\end{align*}
		Therefore, $\|f-f_N\|_{L^p(C_R)}<\epsilon$ if $N>R+\frac{2}{n}+\frac{2}{n}\left[ \frac{4^nB^{(p'-1)}\|f\|_{L^p(\RR^n)}^{p'}C^{p'}R^{n(p'-1)}}{w_\alpha\epsilon^{p'}} \right]^{\frac{1}{\alpha p'-n}}$.
	\end{proof}
	
	\begin{lemma}
		\label{lemma:inftynorm<pnorm}
		If $f\in V(R,\delta)$, then $\|f\|_{L^\infty(C_R)}\leq D\|f\|_{L^p(C_R)},$ where $D=\frac{\sup\limits_{x\in C_R} \|K(x,\cdot)\|_{L^{p'}(\RR^n)}}{(1-\delta)^{\frac{1}{p}}}.$
	\end{lemma}
	\begin{proof}
		Since $V$ is the range of an idempotent integral operator, we have		
		$$f(x)=\int_{\RR^n} f(y)K(x,y)dy, \hspace{0.5cm} \text{for all } f\in V, ~x\in \RR^n.$$
		Then,
		\begin{align*}
		|f(x)|&\leq \int_{\RR^n} |f(y)||K(x,y)|dy\\
		&\leq \|f\|_{L^p(\RR^n)} \|K(x,\cdot)\|_{L^{p'}(\RR^n)}		
		\end{align*}
		\begin{equation}
		\|f\|_{L^\infty(C_R)}\leq \sup\limits_{x\in C_R} \|K(x,\cdot)\|_{L^{p'}(\RR^n)} \|f\|_{L^p(\RR^n)}.
		\label{eqn:nrmLinftyCRbound}
		\end{equation}
		Let $f\in V(R,\delta)$, then we have $(1-\delta)\|f\|_{L^p(\RR^n)}^p\leq \|f\|_{L^p(C_R)}^p$.\\
		Therefore, $$\|f\|_{L^\infty(C_R)}\leq \frac{\sup\limits_{x\in C_R} \|K(x,\cdot)\|_{L^{p'}(\RR^n)}}{(1-\delta)^{\frac{1}{p}}}\|f\|_{L^p(C_R)}.$$
	\end{proof}
	
	The following result is a well-known bound for the number of open balls of fixed radius to cover a closed ball in a finite-dimensional space, see \cite{cucker2007learning}.
	\begin{lemma}
		\label{lemma:finitedimcoveringno}
		Let $X$ be a Banach space of dimension $s$ and $\overline{B(0;r)}$ denotes the closed ball of radius $r$ centered at the origin. Then the number of open balls of radius $\omega$ to cover $\overline{B(0;r)}$ is bounded by $\left( \frac{2r}{\omega}+1 \right)^s$.
	\end{lemma}
	
	\begin{lemma}
		\label{lemma:VRtotallybounded}
		The set $V(R,\delta)$ is totally bounded with respect to $\|\cdot\|_{L^\infty(C_R)}$.
	\end{lemma}
	\begin{proof}
		Let $\epsilon>0$ and $f\in V(R,\delta)$ be given. Then by \hyperref[lemma:p-approx]{Lemma \ref{lemma:p-approx}} and \hyperref[lemma:inftynorm<pnorm]{Lemma \ref{lemma:inftynorm<pnorm}} there exists $f_N$ in a finite-dimensional subspace $V_N$ of $V$ such that $\|f-f_N\|_{L^\infty(C_R)}<\frac{\epsilon}{2}.$ 
		
		Let $\overline{B(0;D+\frac{\epsilon}{2})}$ be a closed ball in $V_N$ with respect to $\|\cdot\|_{L^{\infty}(C_R)}$. We know that  $\overline{B(0;D+\frac{\epsilon}{2})}$ is totally bounded, and   let $\aa(\epsilon)$ be the finite collection of $\frac{\epsilon}{2}$-net for $\overline{B(0;D+\frac{\epsilon}{2})}$.
		
		Since  $\|f\|_{L^{\infty}(C_R)}\leq D$ and $\|f-f_N\|_{L^\infty(C_R)}<\frac{\epsilon}{2}$, we get $f_N\in \overline{B(0;D+\frac{\epsilon}{2})}$. 
		This implies that there exists $\tilde{f}\in \aa(\epsilon)$ such that $\|f_N-\tilde{f}\|_{L^{\infty}(C_R)}< \frac{\epsilon}{2}$, and hence $\|f-\tilde{f}\|_{L^\infty(C_R)}< \epsilon$. 	
		Therefore, the finite set $\aa(\epsilon)$ forms an $\epsilon$-net for $V(R,\delta)$.
		\par
		
	\end{proof}
	\begin{rem} 
	\mbox{~} 
	
		\begin{enumerate}
			\item In the above lemma, we choose $f_N\in V_N$ such that $\|f-f_N\|_{L^p(C_R)}<\frac{\epsilon}{2D},$ and $N>R+\frac{2}{n}+\frac{2}{n}\left[ \frac{4^nB^{(p'-1)}(2CD)^{p'}R^{n(p'-1)}}{w_\alpha\epsilon^{p'}} \right]^{\frac{1}{\alpha p'-n}}$.\\
			In particular, we select $N=R+2+\frac{2}{n}\left[ \frac{4^nB^{(p'-1)}(2CD)^{p'}R^{n(p'-1)}}{w_\alpha\epsilon^{p'}} \right]^{\frac{1}{\alpha p'-n}}$, then dimension of $V_N$ is bounded by
			\begin{align*}
			N^nN_0(\Gamma)&=N_0(\Gamma)\left[ R+2+\frac{2}{n}\left( \frac{4^nB^{(p'-1)}(2CD)^{p'}R^{n(p'-1)}}{w_\alpha\epsilon^{p'}} \right)^{\frac{1}{\alpha p'-n}} \right]^n\\
			&\leq 2^nN_0(\Gamma)\Big[ (R+2)^n+C_1\epsilon^{-\frac{np'}{\alpha p'-n}} \Big]:=d_{\epsilon},
			\end{align*}
			where $N_0(\Gamma)=\sup\limits_{k\in \ZZ^n} \left( k+\left[-\frac{1}{2},\frac{1}{2} \right]^n \right)\cap\Gamma$, and $C_1=\left( \frac{2}{n} \right)^n\left( \frac{4^nB^{(p'-1)}(2CD)^{p'}R^{n(p'-1)}}{w_\alpha} \right)^{\frac{n}{\alpha p'-n}}.$
			
			\item If $N(\epsilon)$ denotes the number of elements in $\aa(\epsilon)$ then
			\begin{equation*}
			N(\epsilon)\leq \Big( 1+\frac{4D+2\epsilon}{\epsilon} \Big)^{d_{\epsilon}}=\exp\left( d_{\epsilon}\log\big(3+\frac{4D}{\epsilon}\big) \right)\leq \exp\left( d_{\epsilon}\log\Big(\frac{8D}{\epsilon}\Big) \right).
			\end{equation*}
		\end{enumerate}
	\end{rem}
	
	\section{Random Sampling} \label{Random Sampling}
	In this section, we define independent random variables on $V$ through random samples and estimate their variance and bound. Later, we use Bernstein's inequality to prove the sampling inequality for the set of functions in $V(R,\delta)$ with high probability.\\
	
	Let $\{x_j : j\in \NN\}$ be a sequence of i.i.d. random variables uniformly distributed over $C_R$. For every $f\in V$, we introduce the random variable
	\begin{equation}
	Z_j(f)=|f(x_j)|^p-\frac{1}{R^n}\int\limits_{C_R} |f(x)|^p dx.
	\label{ranvar}
	\end{equation}
	Then $\{Z_j(f)\}_{j\in \NN}$ is a sequence of independent random variable with expectation $\EE[Z_j(f)]=0$.
	
	\begin{lemma}
		\label{lemma:ranvar}
		Let $f,g\in V(R,\delta)$ and $j\in \NN$. Then the following inequalities hold:
		\begin{enumerate}[label=(\roman*)]
			\item $VarZ_j(f)\leq \frac{1}{R^n} \Big( \supl_{x\in C_R} \|K(x,\cdot)\|_{L^{p'}(\RR^n)} \Big)^p$,
			
			\item $\|Z_j(f)\|_\infty \leq \Big( \supl_{x\in C_R} \|K(x,\cdot)\|_{L^{p'}(\RR^n)} \Big)^p$,
			
			\item $Var(Z_j(f)-Z_j(g))\leq \frac{2p}{R^n} \Big( \supl_{x\in C_R} \|K(x,\cdot)\|_{L^{p'}(\RR^n)} \Big)^{p-1}\|f-g\|_{L^\infty(C_R)}$,
			
			\item $\|Z_j(f)-Z_j(g)\|_\infty \leq p \Big( \supl_{x\in C_R} \|K(x,\cdot)\|_{L^{p'}(\RR^n)} \Big)^{p-1} \|f-g\|_{L^\infty(C_R)}$.
		\end{enumerate}
	\end{lemma}
	\begin{proof}
		\begin{enumerate}[label=(\roman*)]
			\item For random variable $Z_j(f)$ with $\EE(Z_j(f))=0$ and by \eqref{eqn:nrmLinftyCRbound},
			\begin{flalign*}
			VarZ_j(f)&=\EE\left( \left[ |f(x_j)|^p-\EE(|f(x_j)|^p) \right]^2 \right) &&\\
			&=\EE( |f(x_j)|^{2p})-\left[ \EE(|f(x_j)|^p) \right]^2\\
			&\leq \EE( |f(x_j)|^{2p})\\
			&= \frac{1}{R^n}\int_{C_R} |f(x)|^{2p} dx\\
			&\leq \frac{1}{R^n} \|f\|_{L^p(C_R)}^p\|f\|_{L^\infty(C_R)}^p\\
			&\leq \frac{1}{R^n} \Big( \sup\limits_{x\in C_R} \|K(x,\cdot)\|_{L^{p'}(\RR^n)} \Big)^p.
			\end{flalign*}
			
			\item Since $f\in V$ with $\|f\|_{L^p(\RR^n)}=1$ and by \eqref{eqn:nrmLinftyCRbound}, we obtain
			\begin{flalign*}
			\|Z_j(f)\|_\infty &= \supl_{\omega\in \Omega} \Big| |f(x_j(\omega))|^p -\frac{1}{R^n}\int_{C_R} |f(x)|^p dx \Big| &&\\
			&\leq \max\left\{ \|f\|_{L^\infty(C_R)}^p, \frac{1}{R^n}\|f\|_{L^p(C_R)}^p \right\}\\
			&= \|f\|_{L^\infty(C_R)}^p\leq \Big( \supl_{x\in C_R} \|K(x,\cdot)\|_{L^{p'}(\RR^n)} \Big)^p.
			\end{flalign*}
			
			\item Using the same method as in (i), we get
			\begin{flalign*}
			Var(Z_j(f)-Z_j(g))&= \EE\left( [|f(x_j)|^p-|g(x_j)|^p]^2 \right)-\left( \EE(|f(x_j)|^p-|g(x_j)|^p) \right)^2 &&\\
			&\leq \frac{1}{R^n} \int_{C_R} (|f(x)|^p-|g(x)|^p)^2 dx\\
			&\leq \frac{1}{R^n} \int_{C_R} \big{|} |f(x)|^p-|g(x)|^p \big|(|f(x)|^p+|g(x)|^p) dx\\
			&\leq \frac{1}{R^n} \||f|^p-|g|^p\|_{L^\infty(C_R)} \left( \|f\|_{L^p(C_R)}^p+\|g\|_{L^p(C_R)}^p \right)\\
			&\leq \frac{2}{R^n}\|(|f|-|g|)(|f|^{p-1}+|f|^{p-2}|g|+\cdots+|f||g|^{p-2}+|g|^{p-1})\|_{L^\infty(C_R)}\\
			&\leq \frac{2}{R^n} p\max\left\{ \|f\|_{L^\infty(C_R)}, \|g\|_{L^\infty(C_R)} \right\}^{p-1} \|f-g\|_{L^\infty(C_R)}\\
			&\leq \frac{2p}{R^n} \Big( \sup\limits_{x\in C_R} \|K(x,\cdot)\|_{L^{p'}(\RR^n)} \Big)^{p-1} \|f-g\|_{L^\infty(C_R)}.
			\end{flalign*}
			
			\item The estimate follows similarly from (ii).
			\begin{flalign*}
			\|Z_j(f)-Z_j(g)\|_\infty &= \supl_{\omega\in \Omega} \Big| |f(x_j(\omega))|^p-|g(x_j(\omega))|^p -\frac{1}{R^n}\Big( \int\limits_{C_R} (|f(x)|^p-|g(x)|^p) dx \Big) \Big| &&\\
			&\leq \max\left\{ \||f|^p-|g|^p\|_{L^\infty(C_R)},\frac{1}{R^n}\||f|^p-|g|^p\|_{L^1(C_R)} \right\}\\
			&=\||f|^p-|g|^p\|_{L^\infty(C_R)}\\
			&\leq p\Big( \supl_{x\in C_R} \|K(x,\cdot)\|_{L^{p'}(\RR^n)} \Big)^{p-1} \|f-g\|_{L^\infty(C_R)}.
			\end{flalign*}
			The last inequality follows from the estimation in (iii).
		\end{enumerate}
	\end{proof}
	
	In the rest of the paper, we denote $k=\sup\limits_{x\in C_R} \|K(x,\cdot)\|_{L^{p'}(\RR^n)}$.	
	The following Bernstein's inequality plays an important role in \hyperref[thm:mainresult]{Theorem \ref{thm:mainresult}}. 
	\begin{thm}[Bernstein's Inequality \cite{bennett1962probability}]
		Let $Y_j, ~ j=1,2,\cdots,r$ be a sequence of bounded, independent random variable with $\EE Y_j=0$, $VarY_j\leq \sigma^2$, and $\|Y_j\|_\infty\leq M$ for $j=1,2,\cdots,r$. Then 
		\begin{equation}
		P\Big( \Big| \sum_{j=1}^{r} Y_j \Big|\geq \lambda \Big)\leq 2\exp\Big( -\frac{\lambda^2}{2r\sigma^2+\frac{2}{3}M\lambda} \Big).
		\label{eqn:bernsineq}
		\end{equation}
	\end{thm}
	
	\begin{thm}
		Let $\{x_j :j\in\NN \}$ be a sequence of i.i.d. random variables that are drawn uniformly from $C_R=[-R/2,R/2]^n$. Then there exist constants $a,b>0$ depending on $n$, $R$, and $\delta$ such that
		\begin{equation}
		P\Big( \sup_{f\in V(R,\delta)} \Big| \sum_{j=1}^{r} Z_j(f) \Big|\geq \lambda \Big)\leq 2a\exp\Big( -\frac{b}{pk^{p-1}} \frac{\lambda^2}{12rR^{-n}+\lambda} \Big).
		\label{maxprob}
		\end{equation}
	\end{thm}
	\begin{proof}
	The proof follows from the similar idea of Bass and Gr{\"o}chenig \cite{bass2010random}. 
		To determine the required probability, we use Bernstein's Inequality \eqref{eqn:bernsineq} repeatedly on independent random variable $Z_j$. We prove the result in the following steps:\par
		Step 1: Let $f\in V(R,\delta)$.  By \hyperref[lemma:VRtotallybounded]{Lemma \ref{lemma:VRtotallybounded}} we can construct a sequence $\{f_l\}_{l\in \NN}$ such that $f_l\in \aa(2^{-l})$ and $\|f-f_l\|_{L^\infty(C_R)}<2^{-l}$.	Then we write
		\begin{equation}
		Z_j(f)=Z_j(f_1)+\sum_{l=2}^{\infty} \Big( Z_j(f_l)-Z_j(f_{l-1}) \Big).
		\label{eqn:ranvarseries}
		\end{equation}
		Indeed, $s_m(f)=Z_j(f_1)+\sum_{l=2}^{m} \left( Z_j(f_l)-Z_j(f_{l-1}) \right)=Z_j(f_m)$ and
		\begin{align*}
		\|Z_j(f)-Z_j(f_m)\|_{\infty}&\leq pk^{p-1}\|f-f_m\|_{L^\infty(C_R)}\\
		&\rightarrow 0 ~~as~ m\rightarrow \infty.
		\end{align*}
		Now consider the events 
		$$\ee=\Big\{ \sup\limits_{f\in V(R,\delta)} \Big| \sum\limits_{j=1}^{r} Z_j(f) \Big|\geq \lambda \Big\},$$
		$$\ee_1=\Big\{ \exists~ f_1\in \aa\left( \frac{1}{2} \right)~:~ \Big| \sum_{j=1}^{r} Z_j(f_1) \Big|\geq \frac{\lambda}{2} \Big\},$$
		and for $l\geq 2$
		\begin{multline*}
		\ee_l=\bigg\{ \exists~ f_l\in \aa(2^{-l}) ~and~ f_{l-1}\in \aa(2^{-l+1}) ~with\\
		\|f_l-f_{l-1}\|_{L^\infty(C_R)}\leq 3\cdot2^{-l}~:~ \Big| \sum_{j=1}^{r} (Z_j(f_l)-Z_j(f_{l-1})) \Big|\geq \frac{\lambda}{2l^2} \bigg\}.
		\end{multline*}
		\textit{Claim:} If $\sup\limits_{f\in V(R,\delta)} \Big| \sum\limits_{j=1}^{r} Z_j(f) \Big|\geq \lambda$, i.e. $P(\ee)>0$ then one of the events $\ee_l$ hold for $l\geq 1$, i.e. $\ee \subseteq \bigcup\limits_{l=1}^{\infty} \ee_l$.\par
		Suppose for all $l\geq 1$, $P(\ee_l)=0$, then for $f\in V(R,\delta)$ and \eqref{eqn:ranvarseries} we get
		\begin{align*}
		\Big| \sum\limits_{j=1}^{r} Z_j(f)\Big|&\leq \Big| \sum\limits_{j=1}^{r} Z_j(f_1) \Big|+\sum\limits_{l=2}^{\infty} \Big| \sum_{j=1}^{r} (Z_j(f_l)-Z_j(f_{l-1})) \Big|\\
		&< \frac{\lambda}{2}+\sum\limits_{l=2}^{\infty} \frac{\lambda}{2l^2}=\frac{\pi^2}{12}\lambda< \lambda.
		\end{align*}
		This is a contradiction.\par
		Step 2: We compute bound for the probability of the event $\ee_1$. Using Bernstein's inequality \eqref{eqn:bernsineq} for the sequence of independent random variable $Z_j(f_1)$, and the results in \hyperref[lemma:ranvar]{Lemma \ref{lemma:ranvar}} $(i) ~\&~ (ii)$, we get
		\begin{align*}
		P\Big( \Big| \sum_{j=1}^{r} Z_j(f_1) \Big|\geq \frac{\lambda}{2} \Big)&\leq 2\exp\Big( -\frac{\frac{\lambda^2}{4}}{2rR^{-n}k^p+\frac{1}{3}k^p\lambda} \Big)\\
		&=2\exp\left( -\frac{3}{4k^p}\frac{\lambda^2}{6rR^{-n}+\lambda} \right).
		\end{align*}
		Therefore,
		\begin{equation}
		P(\ee_1)\leq 2N\left(\frac{1}{2}\right)\exp\left( -\frac{3}{4k^p}\frac{\lambda^2}{6rR^{-n}+\lambda} \right).
		\label{PE1}
		\end{equation}
		
		Step 3: The bound of the probability of the event $\ee_l$ can be found in a similar way as in Step 2. From \eqref{eqn:bernsineq} and \hyperref[lemma:ranvar]{Lemma \ref{lemma:ranvar}} $(iii) ~\&~ (iv)$, we have
		\begin{align*}
		&P\Big( \Big| \sum_{j=1}^{r} (Z_j(f_l)-Z_j(f_{l-1})) \Big|\geq \frac{\lambda}{2l^2} \Big)\\
		&\leq 2\exp\Big( -\frac{\frac{\lambda^2}{4l^4}}{4rpR^{-n}k^{p-1}\|f_l-f_{l-1}\|_{L^\infty(C_R)}+\frac{1}{3}pk^{p-1}\|f_l-f_{l-1}\|_{L^\infty(C_R)}\frac{\lambda}{l^2}} \Big)\\
		&\leq 2\exp\Big( -\frac{1}{4l^4} \frac{\lambda^2}{(4rR^{-n}+\frac{\lambda}{3l^2})pk^{p-1}3\cdot 2^{-l}} \Big)\\
		&\leq 2\exp\Big( -\frac{2^l}{4l^4}\frac{\lambda^2}{pk^{p-1}(12rR^{-n}+\lambda)} \Big).
		\end{align*}
		Hence, \begin{equation}
		P(\ee_l)\leq 2N(2^{-l})N(2^{-l+1})\exp\Big( -\frac{2^l}{4l^4}\frac{\lambda^2}{pk^{p-1}(12rR^{-n}+\lambda)} \Big) \hspace{1cm} l\geq 2.
		\label{PEl}
		\end{equation}
		In the view of the fact that $N(\epsilon)$ is bounded and
		$$N(\epsilon)\leq \exp\left( 2^nN_0(\Gamma)\left[ (R+2)^n+C_1\epsilon^{-\frac{np'}{\alpha p'-n}} \right]\log{\frac{8D}{\epsilon}} \right),$$
		we have
		\begin{align*}
		N(2^{-l})&\leq \exp\left( 2^nN_0(\Gamma)\left[ (R+2)^n+C_12^{\frac{lnp'}{\alpha p'-n}} \right]\log{2^{l+3}D} \right)\\
		&\leq \exp\left( 2^nN_0(\Gamma)\left[ (R+2)^n+C_12^{\frac{lnp'}{\alpha p'-n}} \right] \big[ (l+3)\log{2}+\log{D} \big] \right)
		\end{align*}
		and similarly, \begin{align*}
		N(2^{-l+1})&\leq \exp\left( 2^nN_0(\Gamma)\left[ (R+2)^n+C_12^{\frac{(l-1)np'}{\alpha p'-n}} \right] \big[ (l+2)\log{2}+\log{D} \big] \right)\\
		&\leq \exp\left( 2^nN_0(\Gamma)\left[ (R+2)^n+C_12^{\frac{lnp'}{\alpha p'-n}} \right] \big[ (l+2)\log{2}+\log{D} \big] \right).
		\end{align*}
		Therefore,
		\begin{equation*}
		N(2^{-l})N(2^{-l+1})\leq \exp\left( 2^nN_0(\Gamma)\left[ (R+2)^n+C_12^{\frac{lnp'}{\alpha p'-n}} \right] \big[ (2l+5)\log{2}+2\log{D} \big] \right).
		\end{equation*}
		Since $(\alpha-\frac{n}{p'})>(n+1)$,
		\begin{align*}
		P(\ee_l)&\leq 2\exp\bigg(2^nN_0(\Gamma)\left[ (R+2)^n+C_12^{\frac{lnp'}{\alpha p'-n}} \right] \big[ (2l+5)\log{2}+2\log{D} \big]\\
		&\hspace{3in}-\frac{2^l}{4l^4}\frac{\lambda^2}{pk^{p-1}(12rR^{-n}+\lambda)} \bigg)\\
		&\leq 2\exp\bigg(2^nN_0(\Gamma)\left[ (R+2)^n+C_12^{\frac{ln}{n+1}} \right] \big[ (2l+5)\log{2}+2\log{D} \big]\\
		&\hspace{3in}-\frac{2^l}{4l^4}\frac{\lambda^2}{pk^{p-1}(12rR^{-n}+\lambda)} \bigg)\\
		&=2\exp\bigg[ 2^{\frac{n+1}{n+2}l}\bigg( 2^nN_0(\Gamma)\Big[ (R+2)^n2^{-\frac{n+1}{n+2}l}+C_12^{-\frac{l}{(n+1)(n+2)}} \Big]\Big[ (2l+5)\log{2}+2\log{D} \Big]\\ &\hspace{3in}-\frac{2^{\frac{l}{n+2}}}{4l^4}\frac{\lambda^2}{pk^{p-1}(12rR^{-n}+\lambda)}\bigg) \bigg]\\
		&=2\exp\bigg[ 2^{\frac{n+1}{n+2}l}\bigg( 2^nN_0(\Gamma)\Big[ (R+2)^n(2l+5)2^{-\frac{n+1}{n+2}l}\log{2}+2(R+2)^n2^{-\frac{n+1}{n+2}l}\log{D}\\
		&\hspace{1.7in}+C_1(2l+5)2^{-\frac{l}{(n+1)(n+2)}}\log{2}+2C_12^{-\frac{l}{(n+1)(n+2)}}\log{D} \Big]\\
		&\hspace{3in}-\frac{2^{\frac{l}{n+2}}}{4l^4}\frac{\lambda^2}{pk^{p-1}(12rR^{-n}+\lambda)} \bigg) \bigg]\\
		&\leq 2\exp\bigg[ 2^{\frac{n+1}{n+2}l}\bigg( 2^nN_0(\Gamma)\Big[ 9(R+2)^n2^{-\frac{2(n+1)}{n+2}}\log{2}+2(R+2)^n2^{-\frac{2(n+1)}{n+2}}\log{D}\\
		&\hspace{1.7in}+C_12(n+1)(n+2)2^{-\frac{2(n+1)(n+2)-5\log{2}}{2(n+1)(n+2)\log{2}}}+C_12^{1-\frac{2}{(n+1)(n+2)}}\log{D} \Big]\\
		&\hspace{3in}-\frac{2^{\frac{4}{\log{2}}}}{4\left[ \frac{4(n+2)}{\log{2}} \right]^4}\frac{\lambda^2}{pk^{p-1}(12rR^{-n}+\lambda)} \bigg) \bigg].
		\end{align*}
		Let \begin{align*}
		c_1&=\frac{2^{\frac{4}{\log{2}}-10}(\log{2})^4}{(n+2)^4},\\
		c_2&=2^nN_0(\Gamma)\Big[ 9(R+2)^n2^{-\frac{2(n+1)}{n+2}}\log{2}+2(R+2)^n2^{-\frac{2(n+1)}{n+2}}\log{D}\\
		&\hspace{1.7in}+C_12(n+1)(n+2)2^{-\frac{2(n+1)(n+2)-5\log{2}}{2(n+1)(n+2)\log{2}}}+C_12^{1-\frac{2}{(n+1)(n+2)}}\log{D} \Big],\\
		\phi&=\frac{\lambda^2}{pk^{p-1}(12rR^{-n}+\lambda)}.
		\end{align*}
		Then $P(\ee_l)\leq 2\exp\left( -2^{\frac{n+1}{n+2}l}(c_1\phi-c_2) \right)$, for $\lambda$ large enough such that $c_1\phi-c_2>0.$\\		
		
		Step 4: Since $\ee \subseteq \bigcup\limits_{l=1}^{\infty} \ee_l$, 	we have 
		\begin{equation}
		P(\ee)\leq \sum_{l=1}^{\infty} P(\ee_l).
		\label{PE}
		\end{equation}
		The series $\sum\limits_{l=2}^{\infty} P(\ee_l)\leq \sum\limits_{l=2}^{\infty} 2\exp\left( -2^{\frac{n+1}{n+2}l}(c_1\phi-c_2) \right)$, and a further upper bound can be obtained by the fact $\sum\limits_{l=2}^{\infty} e^{-u^lv}\leq \frac{1}{uv\log{u}}e^{-uv}$.\\
		Therefore, 
		\begin{align*}
		\sum\limits_{l=2}^{\infty} P(\ee_l)&\leq 2\times \frac{1}{2^{\frac{n+1}{n+2}}(c_1\phi-c_2)\log{2^{\frac{n+1}{n+2}}}}\exp\left( -2^{\frac{n+1}{n+2}}(c_1\phi-c_2) \right)\\
		&=\frac{2^{\frac{1}{n+2}}(n+2)}{(n+1)(c_1\phi-c_2)\log{2}}\exp\left( -2^{\frac{n+1}{n+2}}(c_1\phi-c_2) \right).
		\end{align*}
		Choose $\lambda$ large enough such that $(c_1\phi-c_2)\geq \frac{2^{\frac{1}{n+2}}(n+2)}{(n+1)\log{2}}.$\\
		Then
		\begin{align}
		\sum\limits_{l=2}^{\infty} P(\ee_l)&\leq e^{2^{\frac{n+1}{n+2}}c_2}\exp\{-2^{\frac{n+1}{n+2}}c_1\phi\} \notag\\
		&\leq e^{2^{\frac{n+1}{n+2}}c_2}\exp\big(-2^{\frac{n+1}{n+2}}c_1\frac{\lambda^2}{pk^{p-1}(12rR^{-n}+\lambda)}\big).
		\label{sumPEl}
		\end{align}
		Let $a =\max\Big\{ e^{2^{\frac{n+1}{n+2}}c_2},2N\Big( \frac{1}{2D} \Big) \Big\}$ and
		$b =\min\Big\{ 2^{\frac{n+1}{n+2}}c_1, \frac{3}{4k} \Big\}$.
		Now from \eqref{PE}, \eqref{PE1}, and \eqref{sumPEl} we have
		\begin{align*}
		P(\ee)\leq 2a\exp\Big( -\frac{b}{pk^{p-1}} \frac{\lambda^2}{12rR^{-n}+\lambda} \Big).
		\end{align*}
		This completes the proof.
	\end{proof}
	
	\begin{proof}[Proof of Theorem \ref{thm:mainresult}]
		As mentioned in Section \ref{Covering Number}, it is enough to prove the result for the set $V(R,\delta)$. Put $\lambda=\frac{r\mu}{R^n}$, then $$\ee^c=\Big\{ \sup\limits_{f\in V(R,\delta)} \Big| \sum\limits_{j=1}^{r} Z_j(f) \Big|\leq \frac{r\mu}{R^n} \Big\}.$$
		If $\{ x_j \}$ be a random sample set such that the event $\ee^c$ is true, then
		\begin{gather*}
		\Big| \sum\limits_{j=1}^{r} |f(x_j)|^p-\frac{r}{R^n}\int\limits_{C_R} |f(x)|^p dx \Big|\leq \frac{r\mu}{R^n} ~~~\forall~ f\in V(R,\delta)\\
		\frac{r}{R^n}\int\limits_{C_R} |f(x)|^p dx-\frac{r\mu}{R^n}\leq \sum\limits_{j=1}^{r} |f(x_j)|^p\leq \frac{r}{R^n}\int\limits_{C_R} |f(x)|^p dx+\frac{r\mu}{R^n}\\
		\frac{r}{R^n}(1-\delta)-\frac{r\mu}{R^n}\leq \frac{r}{R^n}\int\limits_{C_R} |f(x)|^p dx-\frac{r\mu}{R^n}\leq \sum\limits_{j=1}^{r} |f(x_j)|^p\leq \frac{r}{R^n}\int\limits_{C_R} |f(x)|^p dx+\frac{r\mu}{R^n}\leq \frac{r(1+\mu)}{R^n}\\
		\frac{r}{R^n}(1-\delta-\mu)\leq \sum\limits_{j=1}^{r} |f(x_j)|^p\leq \frac{r(1+\mu)}{R^n}.
		\end{gather*}
		Hence random sample $\{ x_j \}$ satisfy the above sampling inequality with probability
		\begin{align}
		P(\ee^c)&=1-P(\ee) \notag \\
		&\geq 1-2a\exp\Big( -\frac{b}{pk^{p-1}} \frac{\left(\frac{r\mu}{R^n}\right)^2}{12rR^{-n}+\frac{r\mu}{R^n}} \Big) \notag \\
		P(\ee^c)&\geq 1-2a\exp\left( -\frac{b}{pk^{p-1}R^n} \frac{r\mu^2}{12+\mu} \right).
		\label{eqn:minprob}
		\end{align}
		This completes the proof.
	\end{proof}
	
	\begin{rem}\mbox{~}
	
		\begin{enumerate}
			\item From \eqref{eqn:minprob} one can make the probability close to 1 by taking a sufficiently large sample size.
			
			\item The sampling inequality \eqref{eqn:mainsaminq} is true for sufficiently large $\lambda=\frac{r\mu}{R^n}$ such that $(c_1\phi-c_2)\geq \frac{2^{\frac{1}{n+2}}(n+2)}{(n+1)\log{2}}$, i.e. $$r\geq \frac{pk^{p-1}R^n(12+\mu)}{c_1{\mu}^2}\Big[ \frac{2^{\frac{1}{n+2}}(n+2)}{(n+1)\log{2}}+c_2 \Big]=\oo(R^{2n}).$$
		\end{enumerate}
	\end{rem}

	\section*{Acknowledgement}
	The first author acknowledges Council of Scientific \& Industrial Research for the financial support.
	\bibliographystyle{acm}
	\bibliography{ransamreffinal}
\end{document}